\documentclass[12pt]{amsart}

\usepackage{amssymb,latexsym}

\usepackage{color}

\usepackage{enumerate}

\usepackage[T1]{fontenc}

 \usepackage[french,english]{babel}


\newtheorem{thm}{Theorem}[section]

\newtheorem{lem}[thm]{Lemma}
\newtheorem{pro}[thm]{Proposition}
\theoremstyle{definition}

\numberwithin{equation}{section}

\renewcommand{\leq}{\leqslant}

\renewcommand{\geq}{\geqslant}

\newcommand{\R}{\mathbb{R}}
\newcommand{\C}{\mathbb{C}}
\newcommand{\M}{\mathcal{M}}
\newcommand{\A}{\mathcal{A}}

\newcommand{\re}{\textup{Re}}

\newcommand{\eps}{\varepsilon}

\newcommand{\newabstract}[1]{%
  \par\bigskip
  \csname otherlanguage*\endcsname{#1}%
  \csname captions#1\endcsname
  \item[\hskip\labelsep\scshape\abstractname.]
}

\begin{document}

\baselineskip=17pt

\title[]{Sharp omega results for the divisor and circle problems}

\author{Youness Lamzouri}

\address{
Universit\'e de Lorraine, CNRS, IECL, 
F-54000 Nancy, France}

\email{youness.lamzouri@univ-lorraine.fr}

%\date{}

\begin{abstract} 
We establish omega results for the  divisor and circle problems that are conjecturally sharp, while also determining the sign of the large values obtained. This improves on the work of Soundararajan and on the subsequent independent refinements of Sourmelidis and Mahatab, and gives the first improvement on Hafner's 1981 $\Omega_+$ result for the divisor problem and his $\Omega_-$ result for the circle problem. The main new ingredient is a resonance method which
works directly with the phase appearing in the Vorono\"i summation formula. This is
achieved by replacing the usual positive kernels by a one-sided sectorial
kernel, namely the density of a Gamma distribution, whose Fourier transform lies in a suitable
sector of the complex plane.
\end{abstract}

\subjclass[2020]{Primary 11N56; Secondary 11P21, 11L03}

\maketitle

%%%%%%%%%%%%%%%%%%%%%%%%%%%%%%%%%%%%%%%%%%%%%%%%%%%%%%%%%%%%%%%%%%%%%%%%%%%%%%%%%%

\section{Introduction}

The classical Gauss circle problem asks to determine the smallest $\alpha$ for which
$P(x)\ll_{\eps}x^{\alpha+\eps}$ holds for any $\eps>0$, where $P(x)$ is the remainder term in the asymptotic formula for the number of integer lattice points in a disc of radius $\sqrt{x}$. More precisely 
$$ P(x):=\sum_{n\leq x}r(n)-\pi x, $$
where $r(n)$ is the number of ways of writing $n$ as a sum of two squares. A closely related lattice point problem is the Dirichlet divisor problem which asks to estimate the remainder $\Delta(x)$ in the asymptotic formula for the summatory function of the divisor function, namely
$$ \Delta(x):= \sum_{n\leq x} d(n)- x(\log x+2\gamma-1) , 
$$ where $d(n)=\sum_{m\mid n}1$ denotes the divisor function, and $\gamma$ is the Euler-Mascheroni constant. It is widely believed that $\Delta(x), P(x) \ll_{\eps} x^{1/4+\eps}$, while the current best upper bound is due to Huxley \cite{Hu} and states that both $\Delta(x)$ and $P(x)$ are $\ll_{\eps} x^{131/416+\eps}.$ 
In 1916  Hardy \cite{Har} showed that
$$
\Delta(x)=
\begin{cases}
 \Omega_+\big((x\log x)^{1/4}\log_2 x\big),\\
 \Omega_-\big(x^{1/4}\big),
\end{cases}
\quad \textup{ and } \quad 
P(x)=
\begin{cases}
 \Omega_-\big((x\log x)^{1/4}\big),\\
 \Omega_+\big(x^{1/4}\big),
\end{cases}
$$
where here and throughout we let $\log_k$ denote the $k$-th iterate of the natural logarithm function.
Recall that for a real valued function $f$ and a positive function $g$ the symbol $f=\Omega(g)$ means that
$
  \limsup_{x\to\infty}\frac{|f(x)|}{g(x)}>0.
$
We write $f=\Omega_+(g)$ if
$
  \limsup_{x\to\infty}\frac{f(x)}{g(x)}>0,
$
and $f=\Omega_-(g)$ if
$
\liminf_{x\to\infty}\frac{f(x)}{g(x)}<0.
$
Lastly $f=\Omega_\pm(g)$ if $f=\Omega_+(g)$ and $f=\Omega_-(g)$.

Following Hardy's work, the $\Omega_-$ bound for $\Delta(x)$ and the
$\Omega_+$ bound for $P(x)$ were improved by several authors,
with the strongest results to date due to Corr\'adi and K\'atai \cite{CoKa}.
They proved that, for some positive constant $c$,
$$
  \Delta(x)=\Omega_-\left(
  x^{1/4}
  \exp\left(c(\log_2 x)^{1/4}(\log_3 x)^{-3/4}\right)
  \right),
$$
and obtained a corresponding $\Omega_+$ result for $P(x)$. The first
improvements on Hardy's $\Omega_+$ result for $\Delta(x)$ and $\Omega_-$
result for $P(x)$ were obtained by Hafner \cite{Haf} in 1981. He proved that, for some
positive constants $A$ and $B$,
$$
  \Delta(x)=\Omega_+\left(
  (x\log x)^{1/4}
  (\log_2 x)^{(3+2\log 2)/4}
  \exp\left(-A\sqrt{\log_3 x}\right)
  \right),
$$
and
$$
  P(x)=\Omega_-\left(
  (x\log x)^{1/4}
  (\log_2 x)^{(\log 2)/4}
  \exp\left(-B\sqrt{\log_3 x}\right)
  \right).
$$
 In 2003, Soundararajan \cite{Sound} introduced a different approach, obtaining the
following stronger omega results 
\begin{equation}\label{Eq:SoundDivisor}
  \Delta(x)=\Omega\left(
  (x\log x)^{1/4}
  (\log_2 x)^{\frac34(2^{4/3}-1)}
  (\log_3 x)^{-5/8}
  \right),
\end{equation}
and
\begin{equation}\label{Eq:SoundCircle}
  P(x)=\Omega\left(
  (x\log x)^{1/4}
  (\log_2 x)^{\frac34(2^{1/3}-1)}
  (\log_3 x)^{-5/8}
  \right).
\end{equation}
However, unlike Hafner's estimates, Soundararajan's method does not determine the sign of the large values obtained. More recently, Sourmelidis \cite{Sour} and Mahatab \cite{Ma}   independently refined Soundararajan's argument using the resonance method, improving the exponent of $\log_3 x$ in \eqref{Eq:SoundDivisor} and \eqref{Eq:SoundCircle}  from $-5/8$ to
$-3/8$.

Based on a probabilistic heuristic argument, Soundararajan conjectured that his omega results \eqref{Eq:SoundDivisor} and \eqref{Eq:SoundCircle} represent the true maximal order of $\Delta(x)$ and $P(x)$, respectively, up to a factor of $(\log_2 x)^{o(1)}.$
Recently, the author \cite{Lam} studied the distribution of large values of
$\Delta(x)$ and $P(x)$, showing that in a suitable range they have the same
large deviation behavior as their corresponding probabilistic random models. This led to the following more precise conjecture for the maximal order of these error terms (see \cite[Conjecture 1.1] {Lam}): 
$$ \max_{x\in [1, X]}|\Delta(x)| \asymp (X\log X)^{1/4}(\log_2 X)^{\frac{3}{4}(2^{4/3}-1)},$$
and 
$$ \max_{x\in [1, X]}|P(x)| \asymp (X\log X)^{1/4}(\log_2 X)^{\frac{3}{4}(2^{1/3}-1)}.$$
Our main result proves the lower bounds predicted by this conjecture, removing the remaining
$\log_3 x$ loss in the omega results of Soundararajan \cite{Sound}, 
Sourmelidis \cite{Sour}, and Mahatab \cite{Ma}, while also determining the sign of the large values obtained.
Consequently, this gives the first improvement on the $\Omega_+$ result for
$\Delta(x)$ and the $\Omega_-$ result for $P(x)$ since Hafner's work \cite{Haf}.

\begin{thm}\label{Thm:Omega}
We have
\begin{equation}\label{Eq:OmegaDivisor}
\Delta(x)=\Omega_+\left((x\log x)^{1/4}(\log_2 x)^{\frac34(2^{4/3}-1)}\right),
\end{equation}
and
\begin{equation}\label{Eq:OmegaCircle}
P(x)=\Omega_{-}\left((x\log x)^{1/4}(\log_2 x)^{\frac34(2^{1/3}-1)}\right).
\end{equation}
\end{thm}

Our approach proceeds differently from that of Soundararajan \cite{Sound}, and
from the subsequent improvements of Sourmelidis
\cite{Sour} and  Mahatab \cite{Ma}. To describe our method, we only focus on the case of $\Delta(x)$. By Vorono\"i's summation formula (see  \eqref{Eq:Voronoi} below) it suffices to obtain an omega result for $\re(e^{-i\pi/4} G(x))$, where $G$ is the exponential sum
\begin{equation}\label{Eq:ExpSumDirichlet}
G(x)= \sum_{n\leq T} \frac{d(n)}{n^{3/4}} e\big(2\sqrt{n} x\big),
\end{equation} 
with $T\ll x^A$ for some $A>0$, and where $e(\alpha):=e^{2\pi i \alpha}$. 
Soundararajan's method consists of  three main ingredients. The first is a convolution of $\re(e^{-i\pi/4} G(x))$ with the function $e(\sqrt{N}x) K(\sqrt{N}x)$, where $K(x)= \left(\frac{\sin(\pi x)}{\pi x}\right)^2$ is Fej\'er's kernel. This removes the phase $-\pi/4$, but loses control of the sign of
$\re(e^{-i\pi/4}G(x))$, and produces a weighted version of
\eqref{Eq:ExpSumDirichlet} which is essentially supported on integers $n\asymp N$. Second, he selects a set
$\mathcal M\subset [N/4,9N/4]$ and uses Dirichlet's theorem for Diophantine
approximation, together with several dilates of a well-chosen point, to make
the contribution of the terms in $\mathcal M$ dominate the remaining terms.
Finally, he chooses $\mathcal M$ to consist of integers in $[N/4,9N/4]$ with exactly $\lfloor \lambda \log_2N\rfloor$ distinct prime factors. Optimizing over $\lambda$ yields his omega results \eqref{Eq:SoundDivisor} and 
\eqref{Eq:SoundCircle}. 
Sourmelidis \cite{Sour} and independently Mahatab \cite{Ma} improve
Soundararajan's second step by replacing the Diophantine approximation
argument with the resonance method, exploiting the positivity of the
coefficients $d(n)n^{-3/4}$. This saves a factor of $(\log_3 x)^{1/4}$.
However, their arguments still rely on the first and third ingredients of Soundararajan's method, which are responsible for the loss of the sign and the remaining factor  $(\log_3 x)^{-3/8}$.  %begin with the same localization step, and
%therefore remain restricted to the contribution of integers $n\asymp N$. The
%loss comes from this localization, and is due to the uneven distribution of $d(n)$. Indeed, as observed by Soundararajan, the ideal omega
%result would be 
%$\Delta(x)= \Omega\left(x^{1/4} S(\log x)\right)$
%where $S(M)$ denotes the sum of the largest $M$ terms of the sequence
%$(d(n)n^{-3/4})_{n\g%eq 1}$.

The resonance method is especially effective for sums with positive
coefficients when the off-diagonal terms in the relevant moments cannot be
bounded efficiently. In this setting, one usually works with a ``positive''
kernel, by which we mean a kernel $K$ such that both $K$ and its Fourier transform $\widehat K$ are
non-negative (the Gaussian function is the standard example). %This positivity is used to extract the diagonal contribution. 
This principle appears, for
instance, in the work of Aistleitner, Mahatab and Munsch \cite{AMM} on large values of
$\zeta(1+it)$. Sourmelidis \cite{Sour} and independently Mahatab \cite{Ma} showed that the same philosophy can
also be applied to  trigonometric polynomials with positive coefficients. In particular, they used the  resonance method to obtain optimal lower
bounds for $\re(F(x))$ when
$F(x)=\sum_n a_n e(\lambda_n x)
$
has non-negative coefficients.
The difficulty in the present case is the presence of the factor \(e^{-i\pi/4}\), which destroys the direct positivity argument.  In
Soundararajan's method, this phase is removed by an initial convolution with
Fej\'er's kernel, but this step localizes the sum to integers $n\asymp N$
and loses control of the sign of the original quantity. Our approach proceeds differently, by avoiding this preliminary reduction and applying the resonance
method directly to $\re(e^{i\beta}F(x))$. To make this
possible, we replace the usual positive kernels by a one-sided\footnote{This means that the support of the kernel is contained in $[0, \infty)$.} \emph{sectorial}
kernel. More precisely, we take $K$ to be the probability density function of a Gamma
distribution (see \eqref{Eq:GammaKernel} below) so that
$
  \widehat K(\xi)=(1+2\pi i\xi)^{-\alpha}
$ for some $0<\alpha<1$.
The values of $\widehat K$ lie in a sector of the complex plane, and the
parameter $\alpha$ is chosen so that
$\re(e^{i\beta}\widehat K(\xi))
  \geq \delta\cdot \re\widehat K(\xi)\geq 0, $
for some constant $\delta>0$ and all $\xi\in \R$. This restores the positivity needed in the resonance method without first
removing the phase, and therefore avoids the localization step.

Using this argument we prove the following general result for trigonometric
polynomials with non-negative coefficients, twisted by a phase $e^{i\beta}$
with $\beta\in(-\pi/2,\pi/2)$. Let $f(n) \geq 0$ for all $n \in \mathbb{N}$, such that
$
 \sum_{n\geq 1} f(n) < \infty
$. Let $\{\lambda_n\}_{n\in \mathbb{N}}$ be a sequence of positive real numbers. 
For all $x \in \mathbb{R}$ we define
\begin{equation}\label{Eq:DefTrigonometric}
 F(x) := \sum_{n\geq 1} f(n)e(\lambda_n x).
\end{equation}
\begin{thm}\label{Thm:Main} Let $\beta\in \left(-\frac{\pi}{2},\frac{\pi}{2}\right) \setminus\{0\}$ and $0<r<1$ be fixed. Let $0<\delta<\cos(\beta)$ be a fixed real number and put 
$$ \alpha:= \frac{2}{\pi} \arctan\left(\frac{\cos\beta-\delta}{|\sin\beta|}\right).$$
Let $\M\subset \mathbb{N}$ be a finite set with cardinality $|\M|=M$. Then for any real numbers $3<Y<X$ we have
$$
\max_{x\in [Y,X]}\re\left(e^{i\beta}F(x)\right)
\geq \delta r\sum_{m\in\M}f(m) 
+ O\left( F(0) \left(\frac{1+r}{1-r}\right)^M
\left(\frac{Y\log X}{X}\right)^{\alpha}  \right).
$$
 
\end{thm}
The restriction $\beta\in(-\pi/2,\pi/2)$ is natural. Indeed, without assumptions on
the linear relations among the frequencies $\lambda_n$, one cannot in
general force a trigonometric polynomial to point in an arbitrary direction. The
advantage of Theorem \ref{Thm:Main} is that, for angles in this range, we obtain a lower bound for $\re(e^{i\beta}F(x))$ without restriction on the frequencies $\{\lambda_{m}\}_{m\in \M}$, and
retain control of the sign of the large values. Finally, to deduce Theorem \ref{Thm:Omega} from Theorem \ref{Thm:Main}, it
remains to estimate the sum of the largest $M$ terms of the sequences
$(d(n)n^{-3/4})_{n\geq 1}$ and $(r(n)n^{-3/4})_{n\geq 1}$. This is achieved using a result of
Balasubramanian and Ramachandra \cite{BaRa} who used a variant of the Selberg--Delange method.

%We have chosen to focus on the divisor and circle problems, but our method may be applicable to other error terms with similar
%Vorono\"i-type expansions.

\subsection*{Acknowledgments} The author is supported by a junior chair of the Institut Universitaire de France. Part of this work was completed while the author was visiting the Max Planck Institute for Mathematics in Bonn and the Centre de Recherches Math\'ematiques in Montr\'eal. The author is grateful to both institutions for their hospitality and excellent working conditions. The author would also
like to thank Winston Heap for helpful discussions concerning the resonance method.
%%%%%%%%%%%%%%%%%%%%%%%%%%%%%%%%%%%%%%%%%%%%%%%%%%%%%%%%%%%%%%%%%%%%%%%%%%%%%%%%%%%%%%%%%%%%%%%%%%%%%%%%%%%%%%%%%%%%%%%%%%%%%%%%%%%%
\section{Proof of Theorem \ref{Thm:Main}}

\subsection{The resonance method and the choice of resonator}  Let $\M\subset \mathbb{N}$ be a finite set with cardinality $|\M|=M$. A key point in our argument is the choice of the resonator $R(t)$. In
\cite{Sour}, Sourmelidis uses a trigonometric polynomial
$$
R_1(t)=\sum_{q\leq Q} e(d_q t),
$$
where the frequencies $d_q$ are chosen from suitable linear combinations
$\sum_{m\in \M}\varepsilon_m\lambda_m$ with 
    $\varepsilon_m\in\{0,1\}.$
The resulting resonator is effective, but the method requires a rather
involved combinatorial argument.

On the other hand, Mahatab \cite{Ma} imposes the conditions that $\Lambda=\{\lambda_m : m\in \M\}$ is linearly independent over $\mathbb{Q}$ and $\Lambda\subset [C_1 T, 2T]$ for some constant $0<C_1<2$ where $T$ is large. Inspired by the construction in \cite{AMM}, he uses the resonator 
$$ R_2(t)= \sum_{u\in \mathbb{N}(\Lambda)} \widetilde{r}(u) e(u t), $$
where $\widetilde{r}(u)=e^{-u/2T}$ and $\mathbb{N}(\Lambda)$ is the set of linear combinations of elements in $\Lambda$ with non-negative integer coefficients. Using the linear independence of the $\{\lambda_m\}_{m\in \M}$ and the fact that $\widetilde{r}(u+v)=\widetilde{r}(u)\widetilde{r}(v)$, one  may write $R_2(t)$ as the following product
$$ R_2(t)=\prod_{m\in \M}
        \left(1-\widetilde{r}(\lambda_m)e(\lambda_m t)\right)^{-1}.
$$
This choice is very natural from the point of view of the resonance method.
However, this  construction does not work with general sets $\M$ since the weight $\widetilde{r}(\lambda_m)$ decays with the size of the
frequency, and the useful contribution therefore comes from frequencies
lying in a localized range. 

Instead of these choices, we will use a product resonator with a
fixed weight, which turns out to be better adapted for our method. We define
\begin{equation}\label{Eq:TheResonator}
 R(t) := \prod_{m \in \M} \left(1-r e(\lambda_m t)\right)^{-1},
\end{equation}
where $0<r<1$ is a parameter to be chosen. This choice keeps the multiplicative structure of the
product resonator, but does not impose any localization condition on
the frequencies $\{\lambda_m\}_{m\in\mathcal M}$.  Another useful feature of our method is that, unlike Mahatab's argument, no linear independence
assumption on the frequencies is required. 
Expanding the  product \eqref{Eq:TheResonator} we obtain 
$$
 R(t)=\sum_{v\in \A} a_r(v)e(v t), 
$$
where $\mathcal A$ denotes the set of all linear combinations of the
$\lambda_m$, $m\in\mathcal M$, with non-negative integer coefficients, namely
$$
\mathcal A
:=
\left\{
\sum_{m\in\mathcal M} b(m)\lambda_m
:\ b(m)\in \mathbb Z_{\geq 0}\ \text{for all } m\in\mathcal M
\right\}.
$$
Moreover, if $v\in \A$ has exactly $\ell$ distinct representations
\begin{equation}\label{Eq:RepresentationsA}
 v=\sum_{m\in \M} b_1(m)\lambda_m = \sum_{m\in \M} b_2(m)\lambda_m
= \cdots=\sum_{m\in \M} b_\ell(m)\lambda_m,
\end{equation}
with $b_j(m)\in \mathbb Z_{\geq 0}$ then
\begin{equation}\label{Eq:FormulaAr}
 a_r(v)=\sum_{j=1}^{\ell} r^{\sum_{m\in \M} b_j(m)}.
\end{equation}
We will make use of the following inequality 
\begin{equation}\label{Eq:IneqCoeffA}
a_r(v+ \lambda_n)\geq r a_r(v),
\end{equation}
which holds for all $v\in \A$ and all $n\in \M$. This follows from the simple fact that if $v$ has exactly $\ell$ distinct representations as in \eqref{Eq:RepresentationsA} then the element $v+\lambda_n$ has at least the following
$\ell$ distinct representations
$$
v+\lambda_n
=
\sum_{m\in \mathcal M} \widetilde b_1(m)\lambda_m
=\cdots=
\sum_{m\in \mathcal M} \widetilde b_\ell(m)\lambda_m,
$$
where for all $1\leq j\leq \ell$, we have $\widetilde{b}_j(m)=b_j(m) $ for $ m\neq n$ and
$
\widetilde{b}_j(n)=b_j(n)+1.
$
Therefore, we get
$$
a_r(v+\lambda_n) \geq \sum_{j=1}^{\ell} r^{\sum_{m\in \M} \widetilde{b}_j(m)}
= r\left(\sum_{j=1}^{\ell} r^{\sum_{m\in \M} b_j(m)}\right)
= r a_r(v).
$$

We will deduce Theorem \ref{Thm:Main} from the following proposition once we choose our suitable kernel $K\in L^1(\mathbb{R})$. The construction of  such a kernel will be done in Section \ref{Sec:Kernel}. We will adopt the following convention for the Fourier transform of $K$
$$ \widehat{K}(\xi)= \int_{\R} K(x) e(-\xi x) dx.$$
\begin{pro}\label{Pro:Resonance}
Let $f(n) \geq 0$ for all $n \in \mathbb{N}$, such that
$
 \sum_{n\geq 1} f(n) < \infty
$. Let $\{\lambda_n\}_{n\in \mathbb{N}}$ be a sequence of positive real numbers, and $F$ be defined by \eqref{Eq:DefTrigonometric}. Let $\beta \in \left(-\frac{\pi}{2},\frac{\pi}{2}\right)$ and let $\M\subset \mathbb{N}$ be a finite set. Let $K \in L^1(\mathbb{R})$ be a real valued kernel such that $\textup{supp}(K) \subset [0,\infty)$ and  
\begin{equation}\label{Eq:SectorialKernel}
\re\left(e^{i\beta}\widehat{K}(\xi)\right) \geq \delta \cdot \re\left(\widehat{K}(\xi)\right) \geq 0
\end{equation}
for some positive constant $\delta$ and all $\xi\in\mathbb{R}$. Then for all real numbers $T\geq 1$ we have 
\begin{equation}\label{Eq:MainInequality}
\int_0^\infty \re\left(e^{i\beta}F(x)\right)|R(x)|^2K\left(\frac{x}{T}\right)\,dx
\geq \delta r\left(\sum_{m\in \M}f(m)\right)\cdot
\int_0^\infty |R(x)|^2K\left(\frac{x}{T}\right)\,dx.
\end{equation}
\end{pro} 
\begin{proof}
Let 
$$
I_1:= \int_0^\infty |R(x)|^2K\left(\frac{x}{T}\right)\,dx,
$$
and 
$$
I_2:= \int_0^\infty \re\left(e^{i\beta}F(x)\right)|R(x)|^2K\left(\frac{x}{T}\right)\,dx.
$$
Expanding $F(x)$ and $|R(x)|^2$ we obtain
\begin{equation}\label{Eq:InequalityI2}
 \begin{aligned}
I_2
&= \re\left(e^{i\beta}\int_0^\infty \sum_{n\geq 1} f(n)e(\lambda_n x)
\sum_{u,v\in \A} a_r(u)a_r(v)e((u-v)x)K\left(\frac{x}{T}\right)\,dx\right)\\
&= \sum_{n\geq 1} f(n)\sum_{u,v\in \A} a_r(u)a_r(v)
\re\left(e^{i\beta}\int_0^\infty e\left((\lambda_n+u-v)x\right)K\left(\frac{x}{T}\right)\,dx\right)\\
& =T\sum_{n\geq 1} f(n)\sum_{u,v\in \A} a_r(u)a_r(v)
\re\left(e^{i\beta}\widehat{K}\big((v-u-\lambda_n)T\big)\right)\\
& \geq \delta T\sum_{n\geq 1} f(n)\sum_{u,v\in \A} a_r(u)a_r(v)
\re\left(\widehat{K}\big((v-u-\lambda_n)T\big)\right),
\end{aligned} 
\end{equation}
where we interchanged the sums and integral by absolute convergence, and where the last inequality follows from \eqref{Eq:SectorialKernel}.
Since all the terms on the right hand side are non-negative, we will only keep those $n\geq 1$ such that $n\in \M$. Moreover, for each such $n$ we only keep those $v\in \A$ in the second sum that are of the form $v=w+\lambda_n$ with $w\in \A$. Therefore, by \eqref{Eq:IneqCoeffA} we deduce that
\begin{equation}\label{Eq:PositivityI2}
\begin{aligned}
I_2 &\geq \delta T\sum_{n\in \M}f(n)\sum_{u,w\in \A} a_r(u)a_r(w+\lambda_n)
\re\left(\widehat{K}\left((w-u)T\right)\right)\\
&
\geq \delta r\cdot T\sum_{n\in \M} f(n)\sum_{u,w\in \A} a_r(u)a_r(w)
\re\left(\widehat{K}\big((w-u)T\big)\right).
\end{aligned}
\end{equation}
Now we observe that the inner sum over $(u, w)\in \A^2$ equals 
\begin{align}\label{Eq:IdentityI1}
\re\left(\sum_{u,w\in \A} a_r(u)a_r(w)\widehat{K}\left((w-u)T\right)\right)
&=\re\left(\int_0^\infty \sum_{u,w\in \A} a_r(u)a_r(w)
 e((u-w)Tx)K(x)\,dx\right)\nonumber\\ 
& =\re\left(\int_0^\infty |R(Tx)|^2 K(x)\,dx\right) = \frac{I_1}{T},
\end{align}
by a simple change of variable, and since $K$ is real valued. Combining this identity with  \eqref{Eq:PositivityI2}  completes the proof.
\end{proof}

 \subsection{Choosing the appropriate kernel $K$.}\label{Sec:Kernel} Unlike in the works of Mahatab \cite{Ma}, Sourmelidis \cite{Sour}, and
other applications of the resonance method to $L$-functions, such as
\cite{AMM}, we cannot use kernels whose Fourier transforms are
non-negative. Indeed, if $h\in L^1(\mathbb R)$ is supported in
$[0,\infty)$, then $h$ cannot have a real-valued Fourier transform unless
it is identically zero. Moreover, in  order for the one-sided kernel $K$ to satisfy \eqref{Eq:SectorialKernel}, its convolution operator $A(h)=K* h$ needs to be ``sectorial'', more precisely the range of the  Fourier transform of $K$ must lie inside a ``sector'' of the complex plane, namely
$\{z \in \C : 
|\arg z|<\theta_0\},$ for some $0<\theta_0<\pi.
$ Sectorial operators appear in various areas of mathematics, including fractional calculus, partial differential equations, and spectral theory. 

We will take $K$ to be the probability density function of the Gamma distribution $\Gamma(\alpha,1)$. More precisely, for a parameter $0<\alpha<1$ to be chosen (in terms of the angle $\beta$) we define
\begin{equation}\label{Eq:GammaKernel}
K_\alpha(x):=
\begin{cases}
\dfrac{x^{\alpha-1}e^{-x}}{\Gamma(\alpha)} & \text{if } x>0,\\
0 & \text{otherwise.}
\end{cases}
\end{equation}
Then we have 
$$
\int_\R K_\alpha(x)\,dx=\frac{1}{\Gamma(\alpha)}\int_0^\infty x^{\alpha-1}e^{-x}\,dx=1.
$$
The Fourier transform of $K_{\alpha}$ is
\begin{align*}
\widehat K_\alpha(\xi)
&=\int_{-\infty}^{\infty}K_\alpha(x)e(-\xi x)\,dx =\frac{1}{(1+2\pi i\xi)^\alpha}.
\end{align*}
Since $\arg(1+2\pi i\xi)=\arctan(2\pi \xi)$, we deduce that  
\begin{equation}\label{Eq:SecFourierK}
\arg \widehat K_{\alpha}(\xi)=-\alpha\arctan(2\pi \xi)\in \left(-\frac{\alpha\pi}{2},\frac{\alpha\pi}{2}\right),
\end{equation}
and hence $\re(\widehat{K}_{\alpha}(\xi)) > 0$ for all $\xi\in \R$. Assume that $\beta\in \left(-\frac{\pi}{2},\frac{\pi}{2}\right) \setminus\{0\}$ and let $0<\delta<\cos(\beta)$ be a positive constant. We now choose 
\begin{equation}\label{Eq:ChoiceAlpha}
\alpha= \frac{2}{\pi} \arctan\left(\frac{\cos\beta-\delta}{|\sin\beta|}\right).
\end{equation}
Then $0<\alpha<1$ and $\tan(\alpha \pi/2)= (\cos(\beta)-\delta)/|\sin\beta|$. Writing $\theta_{\xi}= \arg \widehat K_{\alpha}(\xi)$ we obtain 
\begin{align}\label{Eq:SecKernelProof}
 \re\big(e^{i\beta} \widehat K_{\alpha}(\xi)\big) &= \cos(\theta_{\xi})\big(\cos(\beta)-\sin(\beta)\tan(\theta_{\xi})\big)|\widehat{K}_{\alpha}(\xi)|\nonumber\\
 & \geq \delta \cos(\theta_{\xi})|\widehat{K}_{\alpha}(\xi)|= \delta \cdot \re(\widehat{K}_{\alpha}(\xi)),
 \end{align}
since 
$$ |\sin(\beta)\tan(\theta_\xi)|\leq |\sin(\beta)|\tan\left(\frac{\alpha\pi}{2}\right)=\cos(\beta)-\delta, $$
by \eqref{Eq:SecFourierK} and our choice of $\alpha.$

%%%%%%%%%%%%%%%%%%%%%%%%%%%%%%%%%%%%%%%%%%%%%%%%%%%%%%%%%%%%%%%%%%%%%%%%%%%%%%%%%%%%%%%%%%%%%%
\subsection{Proof of Theorem \ref{Thm:Main}} Let $T:=X/(2\log X)$, $R$ be defined by \eqref{Eq:TheResonator} and $K_{\alpha}$ the kernel in \eqref{Eq:GammaKernel}. Note that $|R(x)|\leq (1-r)^{-M}$ for all $x\in \R$ and
$$
K_\alpha(x)\ll
\begin{cases}
 e^{-x} & \text{if } x\geq 1,\\
 x^{\alpha-1} & \text{if  } 0<x<1.
\end{cases}
$$
Moreover, the bound $K_{\alpha}(x)\ll x^{\alpha-1}$ holds for all $x>0$. Therefore we get
\begin{align}\label{Eq:RightTailIntegral}
\int_{X}^{\infty}\re\left(e^{i\beta}F(x)\right)|R(x)|^2K_\alpha\left(\frac{x}{T}\right)\,dx \nonumber
&\ll F(0)(1-r)^{-2M}\int_{X}^{\infty}e^{-\frac{x}{T}}\,dx\\
&\ll F(0)(1-r)^{-2M}T e^{-\frac{X}{T}}\ll  \frac{F(0)}{X}(1-r)^{-2M}.
\end{align}
Moreover, we have
\begin{align}\label{Eq:LeftTailIntegral}
\int_0^{Y}\re\left(e^{i\beta}F(x)\right)|R(x)|^2K_\alpha\left(\frac{x}{T}\right)\,dx\nonumber
&\ll F(0)(1-r)^{-2M}\int_0^{Y}\left(\frac{x}{T}\right)^{\alpha-1}\,dx\\
&\ll F(0)(1-r)^{-2M}T^{1-\alpha}Y^\alpha. 
\end{align}
We now combine the estimates \eqref{Eq:RightTailIntegral} and \eqref{Eq:LeftTailIntegral} with \eqref{Eq:SecKernelProof} and Proposition \ref{Pro:Resonance} to deduce that 
\begin{align}\label{Eq:TruncationResonance}
\int_{Y}^{X}\re\left(e^{i\beta}F(x)\right)|R(x)|^2K_{\alpha}\left(\frac{x}{T}\right)\,dx 
& \geq 
\delta r\cdot \left(\sum_{m\in\M}f(m)\right)\int_0^\infty |R(x)|^2K_{\alpha}\left(\frac{x}{T}\right)\,dx \nonumber\\
& \quad \quad \quad +O\Big(F(0)(1-r)^{-2M}T^{1-\alpha}Y^\alpha\Big).
\end{align}
Furthermore, by the same estimates used to prove \eqref{Eq:RightTailIntegral}
and \eqref{Eq:LeftTailIntegral}, together with the bounds $K_{\alpha}(x)\geq 0$ and 
$|\re(e^{i\beta}F(x))|\leq F(0)$, valid for all $x$, we get
\begin{align}\label{Eq:TruncationResonance2}
&\int_{Y}^{X}\re\left(e^{i\beta}F(x)\right)|R(x)|^2K_{\alpha}\left(\frac{x}{T}\right)\,dx \nonumber\\
&\leq \max_{x\in [Y,X]}\re\left(e^{i\beta}F(x)\right)
\int_{Y}^{X}|R(x)|^2K_{\alpha}\left(\frac{x}{T}\right)\,dx\\
&=  \max_{x\in [Y,X]}\re\left(e^{i\beta}F(x)\right)J_1 + O\Big(F(0)(1-r)^{-2M}T^{1-\alpha}Y^\alpha\Big), \nonumber 
\end{align}
where 
$$ J_1:= \int_0^\infty |R(x)|^2K_{\alpha}\left(\frac{x}{T}\right)\,dx.
$$
To complete the proof we need to derive a lower bound on $J_1$. 
 By \eqref{Eq:IdentityI1} we have 
$$
J_1
= T\sum_{u, v\in\A}a_r(u)a_r(v)\re\left(\widehat K_{\alpha}(T(v-u))\right)
\geq T\sum_{v\in\A}a_r(v)^2 
$$
since $\re\left(\widehat K_{\alpha}(\xi)\right)\geq 0$ for all $\xi\in \R$ and $\widehat K_{\alpha}(0)=1$. Moreover, if $v\in \A$ has exactly $\ell$ distinct representations as in \eqref{Eq:RepresentationsA} then by \eqref{Eq:FormulaAr} we have
$$
a_r(v)^2=
\left(\sum_{j=1}^{\ell}r^{\sum_{m\in\M}b_j(m)}\right)^2
\geq \sum_{j=1}^{\ell}r^{2\sum_{m\in\M}b_j(m)}= a_{r^2}(v).
$$
This gives 
$$
\sum_{v\in\A}a_r(v)^2\geq \sum_{v\in\A}a_{r^2}(v)
= (1-r^2)^{-M},
$$
which implies that 
$$
J_1\geq T(1-r^2)^{-M}.
$$
 Combining this inequality with \eqref{Eq:TruncationResonance} and \eqref{Eq:TruncationResonance2} we deduce that 
$$
\max_{x\in [Y,X]}\re\left(e^{i\beta}F(x)\right)
\geq \delta r\sum_{m\in\M}f(m) 
+E 
$$
where 
$$
E\ll \frac{F(0)(1-r)^{-2M}T^{1-\alpha}Y^\alpha}{T(1-r^2)^{-M}}
\ll F(0) \left(\frac{1+r}{1-r}\right)^M
\left(\frac{Y\log X}{X}\right)^{\alpha}.
$$
This completes the proof.

%%%%%%%%%%%%%%%%%%%%%%%%%%%%%%%%%%%%%%%%%%%%%%%%%%%%%%%%%%%%%%%%%%%%%%%%%%%%%%%%%%%%
\section{Application to the divisor and circle problems}

To prove Theorem \ref{Thm:Omega} we will use the following result of Balasubramanian and Ramachandra \cite{BaRa}, who employed an extension of the Selberg-Delange method to estimate the number of positive integers $n$ such that $ng(n)\leq x$, for certain positive multiplicative functions $g$. 

\begin{lem}[Theorem 1 of \cite{BaRa}]\label{Lem:BaRa}
Let $\kappa>0$ be a real number, and $g$ be a positive multiplicative function such that $g(p)=\kappa$ for all primes $p$ and $g(n)\gg n^{-1/16}.$ Then there exists a positive constant $c_0$ such that 
$$ \left|\left\{n\geq 1: n g(n)\leq x\right\}\right| \sim c_0 x (\log x)^{1/\kappa-1}, \textup{ as } x\to \infty.$$
\end{lem}
As a consequence, we deduce the following lemma. 
\begin{lem}\label{Lem:DivisorCircleFirstTerms}
Let $x$ be large. There exist positive constants $c_1$ and $c_2$ such that 
\begin{equation}\label{Eq:DivisorBaRa}
\left|\left\{n\geq 1: n d(n)^{-4/3}\leq x\right\}\right| \sim c_1 x (\log x)^{2^{4/3}-1}, 
\end{equation}
and
\begin{equation}\label{Eq:CircleBaRa}
\left|\left\{n\geq 1: r(n)>0 \textup{ and } n r(n)^{-4/3}\leq x\right\}\right| \sim c_2 x (\log x)^{2^{1/3}-1}. 
\end{equation}
\end{lem}
\begin{proof}
The asymptotic formula \eqref{Eq:DivisorBaRa} follows readily from Lemma \ref{Lem:BaRa} with the choice $g(n)= d(n)^{-4/3}$, which satisfies the assumptions of Lemma \ref{Lem:BaRa}, since $d(n)\ll_{\eps}n^{\eps}$ for all $\eps>0$, and $g(p)=2^{-4/3}$ for all primes $p$. 

We now prove \eqref{Eq:CircleBaRa}. Although this asymptotic formula  does not follow directly from Lemma \ref{Lem:BaRa}, we will show that with minor modifications, the same variant of the Selberg--Delange method used by Balasubramanian and Ramachandra \cite{BaRa} also yields \eqref{Eq:CircleBaRa}.  Since $r(n)\ll_{\eps} n^{\eps}$, the series
$$
\Phi(s):=\sum_{\substack{n\geq 1\\ r(n)>0}} \frac{1}{(n r(n)^{-4/3})^{s}} 
$$
converges absolutely for $\re(s)>1$. Moreover, we know that $\rho(n)=r(n)/4$ is a multiplicative function and
\begin{equation}\label{Eq:SumSquaresPP}
\rho(p^m)
=
\begin{cases}
1 & \text{if } p=2,\\
\bigl(1+(-1)^m\bigr)/2 & \text{if } p\equiv 3 \pmod 4,\\
m+1 & \text{if } p\equiv 1 \pmod 4.
\end{cases}
\end{equation}
Hence, for $\re(s)>1$ we have 
\begin{align}\label{Eq:EulerProduct}
\Phi(s)&= 4^{4s/3} \sum_{\substack{n\geq 1\\ \rho(n)>0}} \frac{\rho(n)^{4s/3}}{n^{s}} \nonumber\\
& = 4^{4s/3}\left(1-\frac{1}{2^{s}}\right)^{-1}\prod_{p\equiv 3 \bmod 4} \left(1-\frac{1}{p^{2s}}\right)^{-1}\prod_{p\equiv 1 \bmod 4}\left(\sum_{m=0}^{\infty} \frac{(m+1)^{4s/3}}{p^{ms}} \right)\\
& = \mathcal{G}(s) \prod_{p\equiv 1 \bmod 4}\left(1-\frac{1}{p^s}\right)^{-2^{4s/3}} \nonumber
\end{align}
where $\mathcal{G}(s)$ is analytic  in $\re(s)>1/2$. On the other hand, for $\re(s)>1$
we have
$$
\zeta(s)L(s,\chi_4)
=
\left(1-\frac{1}{2^{s}}\right)^{-1}
\prod_{p\equiv 1\bmod 4}\left(1-\frac{1}{p^{s}}\right)^{-2}
\prod_{p\equiv 3 \bmod 4}\left(1-\frac{1}{p^{2s}}\right)^{-1},
$$
where $\chi_4$ is the non-principal character modulo $4$.
Therefore, we deduce that for $\re(s)>1$ we have
\begin{equation}\label{Eq:DirichletSeries}
\Phi(s)= \mathcal{H}(s)\bigl(\zeta(s)L(s,\chi_4)\bigr)^{2^{4s/3-1}}, 
\end{equation}
where $\mathcal{H}(s)$ is analytic in the half plane $\re(s)>1/2.$ We now use the same argument of Balasubramanian and
Ramachandra \cite{BaRa} with $\zeta(s)$ replaced by $\zeta(s)L(s,\chi_4)$ and
with the analytic exponent $z(s)=2^{4s/3-1}$. Since $z(1)=2^{1/3}$, this
gives \eqref{Eq:CircleBaRa} for some constant $c_2>0$.
\end{proof}
Using Lemma \ref{Lem:DivisorCircleFirstTerms} we prove the following proposition, which is a key ingredient in the proof of Theorem \ref{Thm:Omega}. 

\begin{pro}\label{Pro:SumLargestTerms}
For a positive integer $M\geq 1$, let $S(M)$ denote the sum of the largest $M$ values of the sequence
$\{d(n)n^{-3/4}:n\geq 1\}$, and let $L(M)$ denote the sum of the largest
$M$ values of the sequence $\{r(n)n^{-3/4}:r(n)>0\}$. Then, we have
$$
S(M)\asymp M^{1/4}(\log M)^{\frac34(2^{4/3}-1)}
$$
and
$$
L(M)\asymp M^{1/4}(\log M)^{\frac34(2^{1/3}-1)}.
$$
\end{pro}
\begin{proof}

We shall only prove the desired estimate for $S(M)$, since  the proof for $L(M)$ follows along the exact same lines, using
\eqref{Eq:CircleBaRa} instead of \eqref{Eq:DivisorBaRa}. Define
$$
N(y):=\left|\left\{n\geq 1:n^{3/4}/{d(n)}\leq y\right\}\right|.
$$
Since
$n^{3/4}/d(n)\leq y$ is equivalent to 
$n d(n)^{-4/3}\leq y^{4/3}$, 
Lemma \ref{Lem:DivisorCircleFirstTerms} gives
\begin{equation}\label{Eq:NumberSequenceDivisor}
N(y)\sim c_3 y^{4/3}(\log y)^{2^{4/3}-1}
\end{equation}
for some constant $c_3>0$. Furthermore, by partial summation applied to this asymptotic formula we have
\begin{equation}\label{Eq:PartialSummation}
\sum_{n^{3/4}/d(n)\leq y}\frac{d(n)}{n^{3/4}}= \int_{0}^y \frac{d N(t)}{t}= \frac{N(y)}{y}+ \int_0^y\frac{N(t)}{t^2}dt\asymp y^{1/3}(\log y)^{2^{4/3}-1}.
\end{equation}
Ordering the sequence $\{d(n) n^{-3/4}\}_{n\geq 1}$ in decreasing order as follows 
$$
\frac{d(n_1)}{n_1^{3/4}}
\geq
\frac{d(n_2)}{n_2^{3/4}}
\geq \cdots,
$$
we get 
$$S(M)=\sum_{j=1}^M\frac{d(n_j)}{n_j^{3/4}}.$$
We now let $y= n_M^{3/4}/d(n_M)$. Then  observe that $N(y-\eta)< M\leq N(y)$ for any $\eta>0$, and hence we obtain $M\sim c_3 y^{4/3}(\log y)^{2^{4/3}-1}$ by \eqref{Eq:NumberSequenceDivisor}, which implies that 
\begin{equation}\label{Eq:EstimateYDivisor}
y\asymp M^{3/4}(\log M)^{-\frac34(2^{4/3}-1)}.
\end{equation}
Furthermore, for any $\eta>0$ we have 
$$
\sum_{n^{3/4}/d(n)\leq y-\eta}\frac{d(n)}{n^{3/4}}\leq S(M)\leq \sum_{n^{3/4}/d(n)\leq y}\frac{d(n)}{n^{3/4}}.
$$
Combining these bounds with \eqref{Eq:PartialSummation} and \eqref{Eq:EstimateYDivisor} completes the proof.
    
\end{proof}

\begin{proof}[Proof of Theorem \ref{Thm:Omega}]
We only work out the divisor case in detail. By Vorono\"i's summation formula (see Eq. (12.4.4) of Titchmarsh \cite{Ti}) we have for any $\eps>0$
\begin{equation}\label{Eq:Voronoi}
\Delta(x)=\frac{x^{\frac14}}{\pi\sqrt{2}}
\sum_{n\leq N}\frac{d(n)}{n^{\frac34}}
\cos\left(4\pi\sqrt{n x}-\frac{\pi}{4}\right)
+O_{\eps}\left(x^{-\frac14}+\left(\frac{T^2}{x}\right)^{\eps}
+\frac{x^{1+\eps}}{T}\right)
\end{equation}
where $N$ is a positive integer, and  
$
T= 2\pi \sqrt{x(N+ \frac12)}.
$ 

Let $\eps=1/1000$. Choosing $N= \lfloor X^{3+1/10}\rfloor$ where $X$ is large, we deduce that uniformly for $\sqrt{X}\leq x\leq X^3$ we have 
\begin{equation}\label{Eq:approxDivisor}
    \Delta(x^2)= \frac{x^{\frac12}}{\pi\sqrt{2}}
\sum_{n\leq N}\frac{d(n)}{n^{\frac34}}
\cos\left(4\pi\sqrt{n} x-\frac{\pi}{4}\right)
+O\left(x^{1/2-\eps}\right).
\end{equation}
We now define $f(n)=d(n)n^{-3/4}$ if $1\leq n\leq N$ and $f(n)$ equals $0$ otherwise. We let $\lambda_n= 2\sqrt{n}$ for all $n\geq 1$ and define 
$$ F(x)= \sum_{n\leq N} \frac{d(n)}{n^{3/4}} e\left(2\sqrt{n} x\right). 
$$
We order the $f(n)$ in decreasing order as follows 
$f(n_1)\geq f(n_2)\geq f(n_3)\geq \cdots$. Let $M=\lfloor(\log X)/4\rfloor$ and\footnote{Note that since $d(n)=n^{o(1)}$ and $M\ll \log X$, the first $M$ terms of the sequence $\{d(n)n^{-3/4}\}_{n\geq 1}$ all satisfy $n\leq N$.} put $\M=\{n_1, \dots, n_M\}.$ 
Then it follows from Theorem \ref{Thm:Main} with $\beta=-\pi/4$, $r= 1/3$ and $\delta=1/10$,  
that 
$$\max_{\sqrt{X}\leq x\leq X^3} \frac{\Delta(x^2)}{x^{1/2}} \geq \frac{1}{30\pi \sqrt{2}} \sum_{j=1}^M \frac{d(n_j)}{n_j^{3/4}} +O\left(F(0) 2^M \left(\frac{\log X}{X^{5/2}}\right)^{\alpha}+X^{-\eps/2}\right),$$
where $\alpha= (2/\pi)\arctan(1-\sqrt{2}\delta)>0.45.$ Using Proposition \ref{Pro:SumLargestTerms} together with the fact that
$$
F(0)= \sum_{n\leq N}\frac{d(n)}{n^{\frac34}}
\ll N^{\frac14+\eps}
\ll X^{4/5},
$$
completes the proof of \eqref{Eq:OmegaDivisor}.

The proof in the circle case is similar. We use the truncated
Vorono\"i formula for $P(x^2)$, namely
\begin{equation}\label{Eq:VoronoiCircle}
P(x^2)
=
-\frac{x^{1/2}}{\pi}
\sum_{n\leq N}\frac{r(n)}{n^{3/4}}
\cos\left(2\pi\sqrt n\,x+\frac{\pi}{4}\right)
+O\left(x^{1/2-\varepsilon}\right),
\end{equation}
which holds uniformly in the same range $\sqrt{X}\leq x\leq X^3$ after taking $N=\lfloor X^{3+1/10}\rfloor$ as before.
We then apply Theorem \ref{Thm:Main} with
$f(n)=r(n)n^{-3/4}$ if $1\leq n\leq N$ and $f(n)$ equals $0$ otherwise, 
$\lambda_n=\sqrt n$, $\beta=\pi/4$,
and with the same choices $r=1/3$, $\delta=1/10$,
$M=\lfloor(\log X)/4\rfloor$ and $\M$ is the set of indices of the largest $M$ terms of the sequence $\{f(n)\}_{n\geq 1}$. Since \eqref{Eq:VoronoiCircle} contains the
minus sign above, the lower bound for
$\re(e^{i\pi/4}F(x))$ gives a lower bound for $-P(x^2)$. Using the
estimate for $L(M)$ from Proposition \ref{Pro:SumLargestTerms}, together
with the bound
$$
\sum_{n\leq N}\frac{r(n)}{n^{3/4}}\ll N^{1/4+\varepsilon},
$$
completes the proof of \eqref{Eq:OmegaCircle}.
\end{proof}

\end{document}